\documentclass[12pt, reqno]{amsart}
\usepackage{amsmath, amsthm, amscd, amsfonts, amssymb}
\usepackage{bm}

\newtheorem{theorem}{Theorem}[section]
\newtheorem{lemma}[theorem]{Lemma}

\theoremstyle{definition}
\newtheorem{definition}[theorem]{Definition}
\newtheorem{example}[theorem]{Example}

\newtheorem{remark}[theorem]{Remark}
\numberwithin{equation}{section}

\newcommand{\vp}{\varphi}

\newcommand{\clb}{\mathcal{B}}

\newcommand{\cld}{\mathcal{D}}

\newcommand{\clg}{\mathcal{G}}
\newcommand{\clh}{\mathcal{H}}
\newcommand{\clk}{\mathcal{K}}

\newcommand{\clm}{\mathcal{M}}

\newcommand{\cls}{\mathcal{S}}

\newcommand{\D}{\mathbb{D}}

\newcommand{\T}{\mathbb{T}}

\newcommand{\raro}{\rightarrow}

\newcommand{\br}{B_{\sigma,e^{i\theta} }}

\newcommand{\C}{\mathbb{C}}
\newcommand{\he}{H^2_E(\T)}

\begin{document}

\setcounter{page}{1}


\title[Brownian shifts on vector-valued Hardy spaces]{Invariant subspaces of Brownian shifts on vector-valued Hardy spaces}

\author[Das]{Nilanjan Das}
\address{Indian Statistical Institute, Statistics and Mathematics Unit, 8th Mile, Mysore Road, Bangalore, 560059,
India}
\email{nilanjand7@gmail.com}

\author[Das]{Soma Das}
\address{Indian Statistical Institute, Statistics and Mathematics Unit, 8th Mile, Mysore Road, Bangalore, 560059, India}
\email{dsoma994@gmail.com}

\author[Sarkar]{Jaydeb Sarkar}
\address{Indian Statistical Institute, Statistics and Mathematics Unit, 8th Mile, Mysore Road, Bangalore, 560059,
India}
\email{jay@isibang.ac.in, jaydeb@gmail.com}



\begin{abstract}
We characterize invariant subspaces of Brownian shifts on vector-valued Hardy spaces. We also solve the unitary equivalence problem for the invariant subspaces of these shifts.
\end{abstract}

\maketitle

\tableofcontents

\section{Introduction}\label{sec intro}

The computation of the lattices of invariant subspaces of bounded linear operators acting on Hilbert spaces is a consistently fascinating problem. There are few operators for which a complete description of the lattice of invariant subspaces is known \cite{Eva, Hal1, Hal2, Radj-Rosen}. Among the known cases, one of the most notable is the shift operator $S$ on $H^2(\T)$ (where $\T = \partial \D$ and $\D = \{z \in \mathbb{C}: |z| < 1\}$). Here, $H^2(\T)$ denotes the classical Hardy space of analytic functions on the open unit disc $\D$, and
\[
S f = z f,
\]
for all $f \in H^2(\T)$. Recall that $H^2(\T)$ can also be viewed as a closed subspace of $L^2(\T)$ consisting of functions whose Fourier coefficients vanish at all negative indices.

The shift $S$ is a typical example of a pure isometry; that is, an isometry that does not have a unitary part. The invariant subspaces of $S$ are precisely of the form
\[
\varphi H^2(\mathbb{T}),
\]
where $\vp$ is an inner function. This is the classical result of Beurling \cite{Beurling}. Continuing with $S$ on $H^2(\T)$, we now turn to Brownian shifts. These were introduced by Agler and Stankus on $H^2(\T) \oplus \mathbb{C}$ \cite[Definition 5.5]{Agler-Stankus} in the context of $m$-isometries, and, along the lines of Beurling, they characterized the invariant subspaces of such operators. The \textit{Brownian shift} of covariance $\sigma > 0$ and angle $\theta \in [0, 2\pi)$ is the bounded linear operator $\br: H^2(\T)\oplus \mathbb{C}\to H^2(\T) \oplus\mathbb{C}$, defined by
\[
\br = \begin{bmatrix}
S & \sigma (1\otimes 1)\\
0 & e^{i\theta}
\end{bmatrix},
\]
where $((1\otimes 1)\alpha)(z) = \alpha$ for all $\alpha \in \mathbb{C}$ and $z \in \T$. Brownian shifts are connected to the time-shift operators associated with Brownian motion processes \cite{Agler-Stankus}. Also, for a curious statistical take on Brownian shifts, see \cite[page 13]{Jim}.

On the other hand, we recall that shifts on vector-valued Hardy spaces are all examples of pure isometries. Given a Hilbert space $E$, denote by $H^2_E(\T)$ the $E$-valued Hardy space over $\T$. By $S_E$, we refer to the shift operator on $H^2_E(\T)$. When $E = \mathbb{C}$, we simply write $H^2_\mathbb{C}(\T)$ as $H^2(\T)$, and $S_{\mathbb{C}}$ as $S$. In this case, the invariant subspaces of $S_E$ are parameterized by operator-valued inner functions—a classical result known as the Beurling–Lax–Halmos theorem (cf. \cite[Theorem 2.1, p. 239]{Foias-Frazho}). This result lies within the framework of the invariant subspace problem—where, instead of investigating the existence of invariant subspaces for general operators, one fixes a natural operator and seeks a complete description of the lattice of its invariant subspaces.

In this paper, we introduce Brownian shifts on vector-valued Hardy spaces and compute the lattice of invariant subspaces for these operators. This work is inspired both by the classification of invariant subspaces by Agler and Stankus and by the Beurling–Lax–Halmos theorem for shift-invariant subspaces on vector-valued Hardy spaces. In particular, our result yields the lattices of invariant subspaces for the Brownian shifts $\br$ studied by Agler and Stankus and provides a new proof of their structure.

The \textit{Brownian shift on $H^2_E(\T) \oplus E$} of covariance $\sigma > 0$ and angle $\theta \in [0, 2 \pi)$ is the bounded linear operator
\[
B^E_{\sigma, e^{i\theta}} = \begin{bmatrix}
S_E & \sigma i_E\\
0 & e^{i\theta}I_E
\end{bmatrix}: H^2_E(\T) \oplus E \raro H^2_E(\T) \oplus E,
\]
where $i_E : E \raro H_E^2(\T)$ is the inclusion map defined by $(i_E x)(z) = x$ for all $x \in E$ and $z \in \T$.

This is a particular class of Brownian unitaries with positive covariance, as introduced by Agler and Stankus \cite{Agler-Stankus}. Clearly, in the scalar case $E = \mathbb{C}$, we have $B^{\mathbb{C}}_{\sigma, e^{i\theta}} = B_{\sigma, e^{i\theta}}$. We often refer to $B_{\sigma, e^{i\theta}}$ as a Brownian shift on $H^2(\T)$. Throughout the paper, $E$ will denote a fixed, but otherwise arbitrary, separable Hilbert space over $\mathbb{C}$, with the possibility that $E = \mathbb{C}$. One additional motivation for introducing Brownian shifts on vector-valued Hardy spaces is that they are unitarily equivalent to the Brownian shifts on $H^2(\T)$ tensored with identity operators. That is,
\[
B_{\sigma, e^{i\theta}}^E \text{ on } H^2_E(\T) \oplus E \cong B_{\sigma, e^{i\theta}} \otimes I_E  \text{ on } (H^2(\T) \oplus\mathbb{C}) \otimes E,
\]
where \lq\lq\,$\cong$\,\rq\rq denotes unitary equivalence between operators. This also aligns with the unitary equivalence
\[
S_E \text{ on } H^2_E(\T) \cong S \otimes I_E  \text{ on } H^2(\T) \otimes E.
\]

The invariant subspaces of $S_E$, as described in the classical Beurling–Lax–Halmos theorem, now also suggest a similar investigation of the invariant subspaces of the Brownian shifts $B^E_{\sigma, e^{i\theta}}$. In this paper, we do precisely that. To this end, we partition invariant subspaces of Brownian shifts into two distinct types:

\begin{definition}\label{def:inv sub}
Let $\clm$ be a closed subspace of $H^2_E(\T) \oplus E$ that is invariant under $B_{\sigma, e^{i\theta}}^E$. We say that:
\begin{enumerate}
\item $\clm$ is Type I if $\clm \subseteq \he \oplus \{0\}$.
\item $\clm$ is Type II if $\clm \nsubseteq \he \oplus \{0\}$.
\end{enumerate}
\end{definition}

To proceed, we first recall the basic and commonly used terminology: Given a Hilbert space $E_*$, let $H^\infty_{\clb(E_*, E)}(\T)$ denote the Banach space of $\clb(E_*, E)$-valued bounded analytic functions on $\D$, where $\clb(E_*, E)$ is the space of bounded linear operators from $E_*$ to $E$ (we write $H^\infty_{\clb(E_*, E)}(\T)$ as $H^\infty(\T)$ whenever $E_* = E = \mathbb{C}$). Each $\Phi \in H^\infty_{\clb(E_*, E)}(\T)$ induces a multiplication operator $M_\Phi \in \clb(H^2_{E_*}(\T), H^2_{E}(\T))$, defined by
\[
M_\Phi f = \Phi f,
\]
for all $f \in H^2_{E_*}(\T)$. It is important to note that
\[
M_\Phi S_{E_*} = S_{E} M_\Phi.
\]
A function $\Phi \in H^\infty_{\clb(E_*, E)}(\T)$ is called \textit{inner} if $M_\Phi$ is an isometry. This is equivalent to the condition that $\Phi(z)$ is an isometry from $E_*$ to $E$ for almost every $z \in \T$. Recall that for an inner function $\Phi \in H_{\clb(E_*, E)}^\infty(\T)$, the \emph{model space} $\clk_\Phi$ (cf. \cite{NaFo70}) is defined by
\[
\clk_\Phi: = H_{E}^2(\T) \ominus \Phi H_{E_*}^2(\T).
\]
Fix a Brownian shift $B^E_{\sigma, e^{i\theta}}$. Given an inner function $\Phi \in H_{\clb(E_*, E)}^\infty(\T)$, we set
\begin{equation}\label{eqn: G Phi}
\clg_\Phi=\left\{\begin{bmatrix} g \\ y \end{bmatrix} \in H^2_E(\T) \oplus E: y \in E, g = \frac{\Phi x-\sigma y}{z-e^{i\theta}} \in H^2_E(\T) \text{ for some } x \in E_* \right\}.
\end{equation}
This set exhibits some curious features of general interest. For instance, $\clg_\Phi$ is a closed subspace of $H^2_E(\T) \oplus E$ (see Lemma \ref{lemma: G Phi}).

In Theorems \ref{th1-br} and \ref{thm: type 2 complete}, we establish the following invariant subspace theorem for Brownian shifts: Let $\clm$ be a nonzero closed subspace of $\he \oplus E$. Then, the following are true:

\begin{enumerate}
\item $\clm$ is a Type I invariant subspace of $B_{\sigma, e^{i\theta}}^E$ if and only if
\[
\clm = \Phi H^2_{E_1}(\T) \oplus \{0\},
\]
for some inner function $\Phi \in H^\infty_{\clb(E_1,E)}(\T)$ and nonzero Hilbert space $E_1$.

\item $\clm$ is a Type II invariant subspace of $B_{\sigma, e^{i\theta}}^E$ if and only if there exists an inner function $\Phi \in H_{\clb(E_2, E)}^\infty(\T)$ for some nonzero Hilbert space $E_2$, and a nonzero subset $\clg \subseteq \clg_\Phi$ such that
\[
\clm =  \langle \clg \rangle \oplus \left(\Phi H^2_{E_2}(\T) \oplus \{0\}\right).
\]
\end{enumerate}

Given a set $\mathcal{N}$ in a Hilbert space, we denote by $\langle \mathcal{N} \rangle$ the closed linear span of the elements of $\mathcal{N}$.

\begin{definition}\label{def: canonical}
The representations of $\clm$ in (1) and (2) above are referred to as the \textit{canonical representations} of Type I and Type II invariant subspaces of $\br^E$, respectively.
\end{definition}

The inner function $\Phi$ in the canonical representation given in part (2) also exhibits a certain boundary behavior, similar to the scalar case studied by Agler and Stankus. We refer the reader to Remark \ref{rem: boundary} for more details. See also Section \ref{sec: scalar} for a detailed analysis of how the full-length invariant subspace theorem of Agler and Stankus can be recovered using the above result and Remark \ref{rem: boundary}.

Following \cite{DDS} (also see \cite{Ron, DS}), we turn the above invariant subspace result to the problem of unitary equivalence. Specifically, given a pair of closed subspaces $\clm_1$ and $\clm_2$ of $H^2_E(\T) \oplus E$ that are invariant under the Brownian shifts $B_{\sigma_1, e^{i\theta_1}}^E$ and $B_{\sigma_2, e^{i\theta_2}}^E$, respectively, we consider the restriction operators $B^E_{\sigma_1, e^{i\theta_1}}|_{\clm_1}$ on $\clm_1$ and $B^E_{\sigma_2, e^{i\theta_2}}|_{\clm_2}$ on $\clm_2$, and determine when they are unitarily equivalent. In Theorem \ref{thm: unit equiv}, we prove: There exists a unitary $U: \clm_1 \raro \clm_2$ such that
\[
U B^E_{\sigma_1, e^{i\theta_1}}{\big|_{\clm_1}} = B^E_{\sigma_2, e^{i\theta_2}}{\big|_{\clm_2}} U,
\]
if and only if any one of the following conditions is true:
\begin{itemize}
\item[(1)]
Both $\clm_1$ and $\clm_2$ are Type I, and if $\clm_i= \Phi_i H_{E_i}^2(\T)\oplus\{0\}$ for $i=1, 2$, then $\dim E_1 = \dim E_2$.
\item[(2)]
Both $\clm_1$ and $\clm_2$ are Type II. Furthermore, $\theta_1=\theta_2$, and if $\clm_j= \langle \clg_j \rangle \oplus (\Phi_j H_{E_j}^2(\T)\oplus\{0\})$ is the canonical representation of $\clm_j$, $j=1,2$, then there exist a pair of unitaries $U_{\clg}: \langle\clg_1\rangle\raro \langle\clg_2\rangle$ and $U_E: E_1\raro E_2$, such that
\[
U_Ex_1^\prime=x_2^\prime,
\]
whenever
\[
U_{\clg}\begin{bmatrix}
    g_1 \\ y_1
\end{bmatrix}
=\begin{bmatrix}
    g_2\\ y_2
\end{bmatrix},
\]
where $\begin{bmatrix} g_j\\ y_j \end{bmatrix}\in \langle\clg_j\rangle$ with
\[
g_j=\frac{\Phi_jx_j^\prime-\sigma_j y_j}{z-e^{i\theta_j}},
\]
for some unique $x_j^\prime\in E_j$, $j=1, 2$.
\end{itemize}

We refer to Theorem \ref{complex-br} for the scalar version of the above result (also see \cite[Theorem 2.1]{DDS}).

A fundamental problem in Hilbert function space theory is the construction of new operators from existing ones. In particular, a natural question is whether the class of operators obtained by restricting natural operators to their invariant subspaces yields new operators  (see the introductory part of Section \ref{sec: examples} for more details).

In this context, Brownian shifts on $H^2_E(\T) \oplus E$ exhibit mixed behavior, suggesting the potential for even richer structural properties. More specifically, in the final section, we apply our main results to show that there exists an invariant subspace of a Brownian shift such that the corresponding restriction operator is unitarily equivalent to another Brownian shift. At the same time, we also provide examples of invariant subspaces for which the restriction operator is not unitarily equivalent to any Brownian shift. These examples indicate that Brownian shifts exhibit behavior intermediate between that of shifts on the Hardy space—where restrictions to invariant subspaces remain unitarily equivalent shifts—and those on the Bergman and Dirichlet spaces, where such restrictions typically produce unitarily inequivalent operators.

We finally remark that Brownian shifts are integral components of $2$-isometries, which play an important role in understanding the general structure of linear operators (cf. \cite{Badea, Chavan et al, Jan}). In the context of function theory, linear operators, and the prediction problem for $2$-stationary processes, we refer the reader to the concluding part of the paper by Agler and Stankus \cite{AS 2}. For the prediction of stationary processes, we recommend the excellent monograph \cite{BD}, particularly Chapter 2. See also \cite{Sto1, Sto2} for related work and possible directions for future connections.

The remainder of the paper is organized as follows. In Section \ref{sec: Inv sub}, we describe the invariant subspaces of Brownian shifts. Section \ref{sec: Unit equiv} determines when such pairs of invariant subspaces are unitarily equivalent. Section \ref{sec: scalar} presents the results obtained this far in the context of the scalar-valued Hardy space. In particular, it highlights the unitary equivalence result for Brownian shifts on $H^2(\T)$. Section \ref{sec: reducing} presents the structure of reducing subspaces for $\br^E$. The final section, Section \ref{sec: examples}, consists of concrete examples illustrating main results of this paper.

\section{Invariant subspaces}\label{sec: Inv sub}

This section presents a complete description of the invariant subspaces of Brownian shifts $\br^E$ on $H^2_E(\T)\oplus E$. We begin with Type I invariant subspaces. Note that $\clm\subseteq \he \oplus \{0\}$ is a nonzero closed subspace if and only if $\clm = \mathcal{M}_0\oplus \{0\}$ for some nonzero closed subspace $\mathcal{M}_0$ of $\he$.

\begin{theorem}\label{th1-br}
Let $\clm$ be a nonzero closed subspace of $\he \oplus E$. Assume that $\clm \subseteq \he \oplus \{0\}$. Then $B^E_{\sigma, e^{i\theta}}(\clm) \subseteq \clm$ if and only if
\[
\clm = \Phi H^2_{E_1}(\T) \oplus \{0\},
\]
for some inner function $\Phi \in H^\infty_{\clb(E_1,E)}(\T)$ and nonzero Hilbert space $E_1$.
\end{theorem}
\begin{proof}
We know that $\clm = \mathcal{M}_0\oplus \{0\}$ for some nonzero closed subspace $\mathcal{M}_0 \subseteq \he$. Since
\[
B_{\sigma, e^{i\theta}}^E\big|_{\he \oplus \{0\}}=\begin{bmatrix}
S_E & 0\\
0 & 0
\end{bmatrix},
\]
the fact that $\clm$ is invariant under $B_{\sigma, e^{i\theta}}^E\big|_{\he\oplus \{0\}}$ is equivalent to $\clm_0$ being invariant under $S_E$ on $\he$. The classical Beurling-Lax-Halmos theorem (cf. \cite[Theorem 2.1, p. 239]{Foias-Frazho}) guarantees that this is same as saying
\[
\clm_0=\Phi H^2_{E_1}(\T),
\]
for some inner function $\Phi\in H_{\clb(E_1,E)}^\infty(\T)$ and nonzero Hilbert space $E_1$. The converse follows from the upper triangular representation of the Brownian shifts.
\end{proof}

We now proceed to the other type of invariant subspaces. For a fixed Brownian shift $B^E_{\sigma, e^{i\theta}}$ and an inner function $\Phi \in H_{\clb(E_2, E)}^\infty(\T)$, recall the construction of $\clg_\Phi$ from \eqref{eqn: G Phi}:
\[
\clg_\Phi=\left\{\begin{bmatrix} g \\ y \end{bmatrix} \in H^2_E(\T) \oplus E: y \in E, g = \frac{\Phi x-\sigma y}{z-e^{i\theta}} \in H^2_E(\T) \text{ for some } x \in E_2 \right\}.
\]
In the following, we prove that the set $\clg_\Phi$ is special:

\begin{lemma}\label{lemma: G Phi}
Let $\Phi \in H_{\clb(E_2,E)}^\infty(\T)$ be an inner function. If $\begin{bmatrix}
g\\
y
\end{bmatrix} \in \clg_\Phi$ is a nonzero element, then $g \in \clk_\Phi$ and $y \neq 0$. Moreover, $\clg_\Phi$ is a closed subspace of $H^2_E(\T) \oplus E$.
\end{lemma}
\begin{proof}
Suppose $g=\frac{\Phi x-\sigma y}{z-e^{i\theta}}\in \he$ for some $x\in E_2$ and $y\in E$. We have $zg + \sigma y = e^{i\theta}g + \Phi x$. Let $\{f_j:j\geq 1\}$ be an orthonormal basis for $E_2$. Then for any $n\geq 1$ and $j\geq 1$, we have (note that $\Phi(0)^* \in \clb(E, E_2)$)
\[
\langle y,\Phi (z^nf_j) \rangle = \langle \Phi(0)^* y, z^nf_j \rangle = 0,
\]
and similarly, we also have (note that $x \in E_2$) $\langle \Phi x,\Phi (z^nf_j)\rangle = 0$. Then
\[
\langle zg + \sigma y,\Phi (z^nf_j) \rangle = \langle e^{i\theta}g + \Phi x,\Phi (z^nf_j) \rangle,
\]
implies
\[
\langle zg,\Phi (z^nf_j) \rangle + \sigma\langle y,\Phi (z^nf_j) \rangle  = e^{i\theta}\langle g,\Phi (z^nf_j) \rangle +\langle \Phi x,\Phi (z^nf_j)\rangle,
\]
which gives
\begin{equation}\label{cond-1}
\langle M^*_\Phi g, z^{n-1}f_j \rangle = e^{i\theta}\langle M^*_\Phi g,z^nf_j \rangle.
\end{equation}
Let
\[
M^*_\Phi g = \sum_{j=0}^{\infty} x_j z^j,
\]
where $x_j \in E_2$ for all $j \geq 0$. If any one of these $x_j$'s is nonzero, then by \eqref{cond-1} we get
\[
\|M^*_\Phi g\| = \infty,
\]
which is not possible. This shows that $M^*_\Phi g = 0$, thereby proving the claim that $g \in \clk_\Phi$. Next, assume that $y = 0$. Then $zg = \Phi x + e^{i\theta} g$, which implies
\[
\|zg\|^2 = \|\Phi x\|^2 + \|g\|^2,
\]
as $g \in \clk_\Phi$. As $\|zg\| = \|g\|$ and $\|\Phi x\| = \|x\|$, we conclude that $x = 0$, and consequently $g = 0$.

\noindent
Now we turn to prove that $\clg_\Phi$ is a closed subspace. It is easy to see that $\clg_\Phi$ is indeed a subspace. To show that it is closed, we pick a sequence $\left\{\begin{bmatrix}
g_n\\
y_n
\end{bmatrix}\right\}\in \clg_\Phi$ such that
\[
\begin{bmatrix}
g_n\\
y_n
\end{bmatrix}\raro\begin{bmatrix}
g\\
y
\end{bmatrix}\in H^2_E(\T)\oplus E.
\]
Equivalently, $g_n\raro g$ in $H^2_E(\T)$ and $y_n\raro y$ in $E$. Now for each $g_n$, there is $x_n\in E_2$ such that
$$
zg_n+\sigma y_n=\Phi x_n+e^{i\theta} g_n,
$$
which means $\{\Phi x_n\}$ is convergent, and this implies $\{x_n\}$ is a Cauchy sequence in $E_2$. Thus $x_n\raro x\in E_2$, and hence $\Phi x_n\raro \Phi x$.
From the above identity, by passing over to the limit we see that
$$
zg+\sigma y=\Phi x+e^{i\theta} g,
$$
that is, $g=\frac{\Phi x-\sigma y}{z-e^{i\theta}}$. This completes the proof of the lemma.
\end{proof}

Given a closed subspace $\clm \subseteq \he \oplus E$, the \textit{defect space} of $\clm$ is defined by
\[
\cld_\clm = \clm \ominus (\clm \cap (\he \oplus \{0\})).
\]

\begin{theorem}\label{th2-br}
Let $\clm$ be a nonzero closed subspace of $\he \oplus E$. Assume that $\clm \nsubseteq \he \oplus \{0\}$. If $B^E_{\sigma, e^{i\theta}}(\clm) \subseteq \clm$, then there exists an inner function $\Phi \in H_{\clb(E_2, E)}^\infty(\T)$ for some nonzero Hilbert space $E_2$, and a set $\clg \subseteq \clg_\Phi$ such that
\[
\clm = \cld_\clm \oplus \left(\Phi H^2_{E_2}(\T) \oplus \{0\}\right),
\]
and
\[
\cld_\clm = \langle \clg \rangle.
\]
Moreover, if $\begin{bmatrix} g \\ y \end{bmatrix} \in \cld_\clm$, then there exists unique $x \in E_2$ with $\|x\| = \sigma \|y\|$ such that
\[
g = \frac{\Phi x-\sigma y}{z-e^{i\theta}}.
\]
\end{theorem}
\begin{proof}
Set
\[
\clm_0 := \clm \cap (\he \oplus \{0\}),
\]
and decompose $\clm$ as
\begin{equation*}
\clm =  \cld_\clm \oplus \clm_0.
\end{equation*}
Since $\clm\nsubseteq H^2_E(\T)\oplus\{0\}$, $\cld_\clm \neq \{0\}$. We need to show that $\clm_0$ is also a nonzero subspace of $\clm$. If possible, let $\clm_0 = \{0\}$. By the assumption, there exists $$F = \begin{bmatrix}
		f\\
		x^\prime
	\end{bmatrix} \in \clm$$ such that $x^\prime \neq 0$. Let us observe that
\begin{align*}
B_{\sigma, e^{i\theta}}^E\begin{bmatrix}
    f \\ x^\prime
\end{bmatrix}-e^{i\theta}\begin{bmatrix}
    f \\ x^\prime
\end{bmatrix}&=\begin{bmatrix}
    zf+\sigma x^\prime \\ e^{i\theta}x^\prime
\end{bmatrix}-e^{i\theta}\begin{bmatrix}
    f \\ x^\prime
\end{bmatrix} =\begin{bmatrix}
    (z-e^{i\theta})f+\sigma x^\prime \\ 0
\end{bmatrix}
\in \clm_0.
\end{align*}
As $\clm_0 = \{0\}$, we have $(z-e^{i\theta})f+\sigma x^\prime=0$, that is,
$$
\frac{\sigma}{e^{i\theta}-z}x^\prime=f\in H^2_E(\T),
$$
which is a contradiction, as $(e^{i\theta}-z)^{-1}$ is not square integrable over $\T$. Thus, $\clm_0 \neq \{0\}$. As $B_{\sigma, e^{i\theta}}^E (\clm) \subseteq \clm$ and $\clm_0 \subseteq \he \oplus \{0\}$, in view of the upper triangular block matrix representation of $B_{\sigma, e^{i\theta}}^E$, it follows that
\[
B_{\sigma, e^{i\theta}}^E(\clm_0) \subseteq \clm_0.
\]
Theorem \ref{th1-br} ensures the existence of a nonzero Hilbert space $E_2$ and an inner function $\Phi\in H_{\clb(E_2,E)}^\infty(\T)$ such that
\[
\clm_0 = \Phi H^2_{E_2}(\T) \oplus \{0\}.
\]
We can therefore rewrite $\clm$ as
\[
\clm = \cld_\clm \oplus \left(\Phi H^2_{E_2}(\T) \oplus \{0\}\right).
\]
Now for any nonzero $\begin{bmatrix}g \\ y\end{bmatrix} \in \cld_\clm$, it is easy to check that $g\in \clk_\Phi$ and $y\neq 0$ in $E$. Moreover, we have
\[
B^E_{\sigma, e^{i\theta}}\begin{bmatrix}
g \\
y
\end{bmatrix}=\begin{bmatrix}
zg+\sigma y\\
e^{i\theta}y
\end{bmatrix}= e^{i\theta}\begin{bmatrix}
g \\
y
\end{bmatrix} \oplus \begin{bmatrix}
(z-e^{i\theta})g+ \sigma y\\
0
\end{bmatrix} \in \clm ,
\]
which, in particular, implies
\[
(z-e^{i\theta})g+ \sigma y\in \Phi H^2_{E_2}(\T),
\]
and hence, there exists a unique $h\in H^2_{E_2}(\T)$ (note that $\Phi$ is an inner function) such that
\begin{equation}\label{repre_g}
zg + \sigma y = e^{i\theta}g + \Phi h.
\end{equation}
Let $\{f_j:j\geq 1\}$ be an orthonormal basis for $E_2$. Then for any $n\geq 1$ and $j\geq 1$, we have (note that $\Phi(0)^* \in \clb(E, E_2)$)
\[
\langle y,\Phi (z^nf_j) \rangle = \langle \Phi(0)^* y, z^nf_j \rangle = 0,
\]
and (note again that $\Phi$ is an inner function)
\[
\langle \Phi h,\Phi (z^nf_j)\rangle = \langle h, z^nf_j \rangle.
\]
By \eqref{repre_g}, we now have
\[
\begin{split}
\langle  h, z^nf_j \rangle & = \langle \Phi h, \Phi(z^nf_j) \rangle
\\
& = \langle (zg + \sigma y)- e^{i\theta}g, \Phi(z^nf_j) \rangle
\\
& = \langle zg, \Phi(z^nf_j) \rangle - e^{i\theta} \langle g , \Phi(z^nf_j) \rangle
\\
& = \langle g, \Phi(z^{n-1}f_j) \rangle - e^{i\theta} \langle g , \Phi(z^nf_j) \rangle
\\
& = 0,
\end{split}
\]
as $g \in \clk_\Phi$. In other words, $h \in E_2$. Let us rename $h$ by $x$. Since $y \in E$, by \eqref{repre_g}, we have
\[
\|g\|^2 + \sigma^2 \|y\|^2= \|g\|^2 + \|x\|^2,
\]
and consequently $\|x\|= \sigma \|y\|$. Finally, again by \eqref{repre_g}, we have $g = \frac{\Phi x-\sigma y}{z-e^{i\theta}}$, completing the proof of the theorem.
\end{proof}

We emphasize that for a Type II invariant subspace $\mathcal{M}$, as described in the above theorem, both $\cld_\clm$ and $\Phi H^2_{E_2}(\T)$ are nonzero.

The description of invariant subspaces of Brownian shifts on $H^2(\T)$ by Agler and Stankus also includes a certain boundary property of the associated inner functions. A similar result holds in the setting of vector-valued Hardy spaces. However, to establish this, we need to use the identification of $H^2_E(\T)$ with the Hardy space $H^2_E(\D)$ of $E$-valued square-summable analytic functions on $\D$ \cite[Chapter V]{NaFo70}.

\begin{remark}\label{rem: boundary}
We remain with the setting of Theorem \ref{th2-br}. We have
\[
\Phi x  -\sigma y= (z-e^{i\theta})g,
\]
where $g \in \clk_\Phi$, $y \in E \setminus \{0\}$, and $x \in E_2$. We treat $g$ as an element of $H_E^2(\D)$ and write the power series expansion $g(z)= \sum_{n=0}^{\infty} x_nz^n$ on $\D$ (note that $x_n \in E$ for all $n$). Then
\[
\begin{split}
\|g(z)\| \leq \sum_{n=0}^{\infty}\|x_n\||z|^n \leq \left(\sum_{n=0}^{\infty}\|x_n\|^2\right)^{\frac{1}{2}}\left(\sum_{n=0}^{\infty}|z|^{2n}\right)^{\frac{1}{2}} = \|g\|_2 \frac{1}{\sqrt{1-|z|^2}},
\end{split}
\]
for all $z\in\D$. Therefore, we have
\[
\|\Phi(z) x  -\sigma y\| \leq |z-e^{i\theta}|\frac{\|g\|_2}{\sqrt{1-|z|^2}},
\]
which after putting $z =re^{i\theta}$ becomes
\[
\|\Phi(re^{i\theta}) x  -\sigma y\| \leq \|g\|_2 \sqrt{\frac{1-r}{1+r}}.
\]
Finally, letting $r\to 1-$ we get that $\Phi(e^{i\theta}) x:= \lim_{r\to 1-}\Phi(re^{i\theta}) x$ exists and
\[
\Phi(e^{i\theta}) x = \sigma y.
\]
In particular, if $E = \mathbb{C}$, then $\Phi(e^{i\theta})$ refers to the existence of the radial limit value of $\Phi$ at $e^{i\theta}$.
\end{remark}

Recall the following observation from the setting of Lemma \ref{lemma: G Phi}:
\[
\langle \clg \rangle \perp \Phi H^2_{E_2}(\T)\oplus\{0\},
\]
for all nonzero $\clg \subseteq \clg_\Phi$. Moreover, if $\begin{bmatrix}g \\ y \end{bmatrix}\in \clg$ is nonzero, then $y \neq 0$. In view of this, we now present a converse to Theorem \ref{th2-br}:

\begin{theorem}\label{thm: type II conv}
Let $\Phi \in H_{\clb(E_2,E)}^\infty(\T)$ be an inner function. Then $\langle \clg \rangle \oplus \left(\Phi H^2_{E_2}(\T) \oplus \{0\}\right)$ is a Type II invariant subspace of $B_{\sigma, e^{i\theta}}^E$ for every nonzero subset $\clg$ of $\clg_\Phi$.
\end{theorem}
\begin{proof}
For any $f\in H^2_{E_2}(\T)$, we have
\[
B_{\sigma, e^{i\theta}}^E\begin{bmatrix}
\Phi f\\
0
\end{bmatrix}= \begin{bmatrix}
\Phi zf\\
0
\end{bmatrix}\in \Phi H^2_{E_2}(\T) \oplus \{0\},
\]
and for $\begin{bmatrix}g \\ y \end{bmatrix}\in \clg$,
\begin{align*}
B_{\sigma, e^{i\theta}}^E\begin{bmatrix}
g\\
y
\end{bmatrix}= \begin{bmatrix}
zg+ \sigma y\\
e^{i\theta}y
\end{bmatrix}=e^{i\theta}\begin{bmatrix}
g \\
y
\end{bmatrix} + \begin{bmatrix}
(z-e^{i\theta})g+ \sigma y\\
0
\end{bmatrix}
 = e^{i\theta}\begin{bmatrix}
g \\
y
\end{bmatrix} + \begin{bmatrix}
\Phi x\\
0
\end{bmatrix},
\end{align*}
for some $x\in E_2$. Therefore,
$$
B_{\sigma, e^{i\theta}}^E(\textrm{span}(\clg))\subseteq\clm.
$$
Since $B_{\sigma, e^{i\theta}}^E$ is a bounded linear operator, $\clm$ is a closed subspace, we have $B_{\sigma, e^{i\theta}}^E(\langle \clg \rangle) \subseteq \clm$. This proves that $\langle \clg \rangle \oplus \left(\Phi H^2_{E_2}(\T) \oplus \{0\}\right)$ is a Type II invariant subspace of $B_{\sigma, e^{i\theta}}^E$.
\end{proof}

Summarizing Theorems \ref{th2-br} and \ref{thm: type II conv}, we obtain the following characterization of Type II invariant subspaces of Brownian shifts:

\begin{theorem}\label{thm: type 2 complete}
Let $\clm \nsubseteq \he \oplus \{0\}$ be a nonzero closed subspace of $\he \oplus E$. Then $\clm$ is invariant under $B^E_{\sigma, e^{i\theta}}$ if and only if there exists an inner function $\Phi \in H_{\clb(E_2, E)}^\infty(\T)$ for some nonzero Hilbert space $E_2$, and a nonzero subset $\clg \subseteq \clg_\Phi$ such that
\[
\clm =  \langle \clg \rangle \oplus \left(\Phi H^2_{E_2}(\T) \oplus \{0\}\right).
\]
\end{theorem}

The results of this section, when specialized to the case $E= \mathbb{C}$ (that is, the scalar case), recover the results of Agler and Stankus. We will elaborate on this in Section \ref{sec: scalar}.

\section{Equivalent invariant subspaces}\label{sec: Unit equiv}

By equivalent invariant subspaces, we mean unitarily equivalent invariant subspaces of Brownian shifts. In other words, given a pair of nonzero invariant subspaces $\clm_1$ and $\clm_2$, we consider the restriction operators $B^E_{\sigma_1, e^{i\theta_1}}{\big|_{\clm_1}}$ and $B^E_{\sigma_2, e^{i\theta_2}}{\big|_{\clm_2}}$, and ask: when is
\[
B^E_{\sigma_1, e^{i\theta_1}}{\big|_{\clm_1}} \cong B^E_{\sigma_2, e^{i\theta_2}}{\big|_{\clm_2}}?
\]
The following theorem provides a concrete answer in terms of inner functions. At this point, it is convenient to recall the definition of the canonical representations of invariant subspaces of $\br^E$, as given in Definition \ref{def: canonical}.

\begin{theorem}\label{thm: unit equiv}
Fix angles $\theta_1, \theta_2 \in [0, 2 \pi)$ and covariances $\sigma_1, \sigma_2 > 0$. Let $\clm_1$ and $\clm_2$ be nonzero closed invariant subspaces of the Brownian shifts $B^E_{\sigma_1, e^{i\theta_1}}$ and $B^E_{\sigma_2, e^{i\theta_2}}$, respectively. Then
\[
B^E_{\sigma_1, e^{i\theta_1}}{\big|_{\clm_1}} \cong B^E_{\sigma_2, e^{i\theta_2}}{\big|_{\clm_2}},
\]
if and only if any one of the following conditions is true:
\begin{itemize}
\item[(1)]
Both $\clm_1$ and $\clm_2$ are Type I, and if $\clm_i= \Phi_i H_{E_i}^2(\T)\oplus\{0\}$ for $i=1, 2$, then $\dim E_1 = \dim E_2$.
\item[(2)]
Both $\clm_1$ and $\clm_2$ are Type II. Furthermore, $\theta_1=\theta_2$, and if $\clm_j= \langle \clg_j \rangle \oplus(\Phi_j H_{E_j}^2(\T)\oplus\{0\})$ is the canonical representation of $\clm_j$, $j=1,2$, then there exist a pair of unitaries $U_{\clg}: \langle \clg_1 \rangle\raro \langle \clg_2 \rangle$ and $U_E: E_1\raro E_2$, such that
\[
U_Ex_1^\prime=x_2^\prime,
\]
whenever
\[
U_{\clg}\begin{bmatrix}
    g_1 \\ y_1
\end{bmatrix}
=\begin{bmatrix}
    g_2\\ y_2
\end{bmatrix},
\]
where $\begin{bmatrix} g_j\\ y_j \end{bmatrix}\in \langle \clg_j \rangle$ with
\[
g_j=\frac{\Phi_jx_j^\prime-\sigma_j y_j}{z-e^{i\theta_j}},
\]
for some unique $x_j^\prime\in E_j$, $j=1,2$.
\end{itemize}
\end{theorem}
\begin{proof}
Let us start with the proof of the ``if'' part. Provided that condition (1) holds, we assume $\clm_i=\Phi_i H_{E_i}^2(\T)\oplus\{0\}$ for some inner functions $\Phi_i\in H_{\clb(E_i, E)}^\infty(\T)$, $i=1, 2$, with $\dim E_1 = \dim E_2$. Then there exists a unitary operator $\hat{U}$ between $E_1$ and $E_2$. Consider the operator $U:\clm_1\to \clm_2$, defined by
\[
U\begin{bmatrix} \Phi_1 h\\ 0\end{bmatrix}=\begin{bmatrix} \Phi_2 \tilde{U}h\\ 0\end{bmatrix}
\]
for all $h \in H_{E_1}^2(\T)$, where $\tilde{U}: H_{E_1}^2(\T)\raro H_{E_2}^2(\T)$ is the unitary operator induced by $\hat{U}$, acting on $h$ as follows:
\begin{equation}\label{Brown_v_1}
\tilde{U} \tilde h=\sum_{n=0}^\infty z^n\hat{U}\tilde{x}_n,
\end{equation}
for all $\tilde h=\sum_{n=0}^\infty \tilde{x}_n z^n\in H_{E_1}^2(\T)$. Note that
\[
z^k\tilde{U}h=\tilde{U}z^kh,
\]
for all $k\geq 0$. It is clear from the construction that $U$ is a surjective isometry between $\clm_1$ and $\clm_2$ and therefore, is unitary. Moreover,
\begin{align*}
B^E_{\sigma_2, e^{i\theta_2}}U\begin{bmatrix} \Phi_1 h\\ 0\end{bmatrix}&
= \begin{bmatrix} z\Phi_2\tilde{U}h\\ 0\end{bmatrix}
=\begin{bmatrix}
    \Phi_2\tilde{U}(zh)\\ 0\end{bmatrix}
= U\begin{bmatrix} \Phi_1 zh\\ 0\end{bmatrix}
= UB^E_{\sigma_1, e^{i\theta_1}}\begin{bmatrix} \Phi_1 h\\ 0\end{bmatrix},
\end{align*}
which implies $B^E_{\sigma_1, e^{i\theta_1}}{\big|_{\clm_1}} \cong B^E_{\sigma_2, e^{i\theta_2}}{\big|_{\clm_2}}$. Next, we assume that (2) is true. Set
\[
\theta=\theta_1=\theta_2.
\]
Define a linear operator $U:\clm_1\to \clm_2$ by
\[
U\left(\begin{bmatrix}\Phi_1 h \\ 0\end{bmatrix} + \begin{bmatrix} g_1\\ y_1 \end{bmatrix}\right)
= \begin{bmatrix}\Phi_2 \tilde{U}_Eh \\ 0\end{bmatrix} + U_\clg \begin{bmatrix} g_1\\ y_1\end{bmatrix},
\]
for $h\in H_{E_1}^2(\T)$ and $\begin{bmatrix}
g_1\\ y_1\end{bmatrix}\in \langle \clg_1 \rangle$, where $\tilde{U}_E$ is the unitary operator from $H_{E_1}^2(\T)$ to $H_{E_2}^2(\T)$, induced by $U_E$ in the same way as in $(\ref{Brown_v_1})$. It is evident from the construction that $U$ is surjective. In addition, since $U_\clg\begin{bmatrix} g_1 \\ y_1\end{bmatrix}\in \langle \clg_2 \rangle$, we have
\begin{align*}
\left\|\begin{bmatrix}\Phi_2 \tilde{U}_{E}h \\ 0\end{bmatrix} + U_\clg\begin{bmatrix} g_1\\ y_1\end{bmatrix}\right\|^2
=\|h\|^2+\|g_1\|^2+\|y_1\|^2=\left\|\begin{bmatrix}\Phi_1 h \\ 0\end{bmatrix}+\begin{bmatrix} g_1\\ y_1\end{bmatrix}\right\|^2.
\end{align*}
This implies that $U$ is an isometry, and therefore, is a unitary operator as well. For notational simplicity, set
\[
F = \begin{bmatrix}\Phi_1 h \\ 0\end{bmatrix} +  \begin{bmatrix} g_1\\ y_1\end{bmatrix}.
\]
Assuming $U_\clg\begin{bmatrix}
    g_1 \\ y_1
\end{bmatrix}=\begin{bmatrix}
    g_2 \\ y_2
\end{bmatrix}$, we observe that
\begin{align*}
B^E_{\sigma_2, e^{i\theta}}U F
&=B^E_{\sigma_2, e^{i\theta}}\left(\begin{bmatrix}\Phi_2 \tilde{U}_Eh \\ 0\end{bmatrix}+\begin{bmatrix} g_2\\ y_2\end{bmatrix}\right)
=\begin{bmatrix}z\Phi_2 \tilde{U}_Eh \\ 0\end{bmatrix}+\begin{bmatrix} zg_2+\sigma_2y_2\\ e^{i\theta}y_2\end{bmatrix}.
\end{align*}
At this point, we recall that $g_j(z)=\frac{\Phi_j x_j^\prime-\sigma_j y_j}{z-e^{i\theta_j}}$ for a unique $x_j^\prime\in E_j$, $j=1, 2$. This yields
\[
z g_j + \sigma_j y_j = e^{i\theta} g_j + \Phi_jx_j^\prime,
\]
for $j=1,2$. Therefore, recalling that $\tilde{U}_Ex_1^\prime=U_Ex_1^\prime=x_2^\prime$ whenever $U_\clg\begin{bmatrix}
    g_1 \\ y_1
\end{bmatrix}=\begin{bmatrix}
g_2 \\ y_2
\end{bmatrix}$, and that $z^k\tilde{U}_Eh=\tilde{U}_E z^kh$ for any $k\geq0$, we have
\begin{align*}
B^E_{\sigma_2, e^{i\theta}}U F & = \begin{bmatrix}z\Phi_2 \tilde{U}_Eh \\ 0\end{bmatrix}+\begin{bmatrix} zg_2+\sigma_2y_2\\ e^{i\theta}y_2\end{bmatrix}
\\
&=U\left(\begin{bmatrix}\Phi_1 zh \\ 0\end{bmatrix}+e^{i\theta}\begin{bmatrix} g_1\\ y_1\end{bmatrix}\right)+\begin{bmatrix}\Phi_2\tilde{U}_E x_1^\prime\\ 0\end{bmatrix}
\\
&=U\left(\begin{bmatrix}\Phi_1 zh \\ 0\end{bmatrix}+e^{i\theta}\begin{bmatrix} g_1\\ y_1\end{bmatrix}+\begin{bmatrix}\Phi_1x_1^\prime\\ 0\end{bmatrix}\right)\\
&=U\left(\begin{bmatrix}z\Phi_1 h \\ 0\end{bmatrix}+\begin{bmatrix} zg_1+\sigma_1y_1\\ e^{i\theta}y_1\end{bmatrix}\right)\\
&=UB^E_{\sigma_1, e^{i\theta}}\left(\begin{bmatrix}\Phi_1 h \\ 0\end{bmatrix}+\begin{bmatrix} g_1\\ y_1\end{bmatrix}\right),
\end{align*}
that is,
\[
B^E_{\sigma_2, e^{i\theta}}U F = UB^E_{\sigma_1, e^{i\theta}} F,
\]
which again ensures that $B^E_{\sigma_1, e^{i\theta_1}}{\big|_{\clm_1}} \cong B^E_{\sigma_2, e^{i\theta_2}}{\big|_{\clm_2}}$. We now turn to the converse. For simplicity, we divide the proof into three parts:

\textsf{Step 1:} We claim that if we assume, without loss of generality, that $\clm_1=\Phi H_{E^\prime}^2(\T)\oplus\{0\}$ is of Type I, and $\clm_2=  \langle\clg\rangle\oplus(\Psi H_{E^{\prime\prime}}^2(\T)\oplus\{0\})$ is of Type II, for certain inner functions $\Phi$ and $\Psi$, then $B^E_{\sigma_1, e^{i\theta_1}}{\big|_{\clm_1}}$ and $B^E_{\sigma_2, e^{i\theta_2}}{\big|_{\clm_2}}$ are not unitarily equivalent. Indeed, if it is not true, then, in particular, we will have the norm identity
\[
\left\|B^E_{\sigma_1, e^{i\theta_1}}{\big|_{\clm_1}}\right\|= \left\|B^E_{\sigma_2, e^{i\theta_2}}{\big|_{\clm_2}}\right\|.
\]
However, for any $h\in H_{E^\prime}^2(\T)$, we have
\begin{align*}
\left\|B^E_{\sigma_1, e^{i\theta_1}}\begin{bmatrix} \Phi h\\ 0\end{bmatrix}\right\| =\left\|\begin{bmatrix} z\Phi h\\ 0\end{bmatrix}\right\|
=\left\|\begin{bmatrix} \Phi h\\ 0\end{bmatrix}\right\|,
\end{align*}
that is, $\left\|B^E_{\sigma_1, e^{i\theta_1}}{\big|_{\clm_1}}\right\|=1$. On the other hand, we have
\begin{align*}
\left\|B^E_{\sigma_2, e^{i\theta_2}}\begin{bmatrix} g\\ y\end{bmatrix}\right\|^2 =\left\|\begin{bmatrix} zg+\sigma_2y\\ e^{i\theta_2}y\end{bmatrix}\right\|^2
= \|y\|^2+\|g\|^2+\sigma_2^2\|y\|^2
= \left\|\begin{bmatrix} g\\
y\end{bmatrix}\right\|^2+\sigma_2^2\|y\|^2 >\left\|\begin{bmatrix} g\\
y\end{bmatrix}\right\|^2,
\end{align*}
for any nonzero $\begin{bmatrix}g \\ y\end{bmatrix}\in \langle\clg\rangle$ (recall that $y\neq 0$ in this case), and hence
\[
\left\|B^E_{\sigma_2, e^{i\theta_2}}{\big|_{\clm_2}}\right\|>1,
\]
which leads to a contradiction. Therefore, $\clm_1$ and $\clm_2$ must be of same type.

\textsf{Step 2:} Suppose both of them are Type I with the canonical decomposition
\[
\clm_i=\Phi_i H_{E_i}^2(\T)\oplus\{0\},
\]
for some inner function $\Phi_i\in H^\infty_{\clb(E_i, E)}(\T)$, $i=1, 2$. Now according to our hypothesis, there exists a unitary operator $U: \clm_1\raro\clm_2$ such that $UB^E_{\sigma_1, e^{i\theta_1}}\big|_{\clm_1}= B^E_{\sigma_2, e^{i\theta_2}}\big|_{\clm_2}U$. First, let us observe that for $h_1 \in H^2_{E_1}(\T)$, if
$$
U\begin{bmatrix}\Phi_1h_1 \\0\end{bmatrix}=\begin{bmatrix}
    \Phi_2h_2 \\ 0
\end{bmatrix},
$$
for some $h_2 \in H^2_{E_2}(\T)$, then
$$
U\begin{bmatrix}z^n\Phi_1 h_1 \\0\end{bmatrix}
= U\left(B^E_{\sigma_1, e^{i\theta_1}}\right)^n\begin{bmatrix}\Phi_1 h_1 \\0\end{bmatrix} = \left(B^E_{\sigma_2, e^{i\theta_2}}\right)^n U \begin{bmatrix}\Phi_1 h_1 \\0\end{bmatrix}
= \left(B^E_{\sigma_2, e^{i\theta_2}}\right)^n\begin{bmatrix}
\Phi_2 h_2\\ 0
\end{bmatrix},
$$
that is,
\[
U\begin{bmatrix}z^n\Phi_1 h_1 \\0\end{bmatrix} = \begin{bmatrix}z^n\Phi_2 h_2 \\0
\end{bmatrix},
\]
for all $n\geq 0$. Now, we take any $x_1^{\prime\prime}\in E_1$, and assume that
$$
U\begin{bmatrix}\Phi_1 x_1^{\prime\prime} \\0\end{bmatrix}=\begin{bmatrix}
    \Phi_2 h^\prime \\ 0
\end{bmatrix},
$$
for some $h^\prime(z)=\sum_{n=0}^\infty z^n\tilde{x}_n^\prime\in H_{E_2}^2(\T)$. Suppose $\tilde{x}_k^\prime\neq 0$ for some $k\geq 1$, and also suppose
$U\begin{bmatrix}
    \Phi_1 g_k \\0
\end{bmatrix}
=\begin{bmatrix}
   \Phi_2 \tilde{x}_k^\prime \\0
\end{bmatrix}$ for some $g_k\in H_{E_1}^2(\T)$. Therefore, $\|\tilde{x}_k^\prime\|^2=\langle \Phi_2h^\prime, \Phi_2z^k \tilde{x}_k^\prime\rangle$ implies
\begin{align*}
\|\tilde{x}_k^\prime\|^2 = \left\langle \begin{bmatrix}\Phi_2 h^\prime \\0
\end{bmatrix}, \begin{bmatrix} z^k\Phi_2\tilde{x}_k^\prime \\0
\end{bmatrix} \right\rangle = \left\langle U\begin{bmatrix}\Phi_1 x_1^{\prime\prime} \\0\end{bmatrix}, U\begin{bmatrix}
    z^k\Phi_1g_k \\0
\end{bmatrix} \right\rangle = \left\langle \begin{bmatrix}\Phi_1 x_1^{\prime\prime} \\0\end{bmatrix}, \begin{bmatrix}
    \Phi_1z^kg_k \\0
\end{bmatrix} \right\rangle,
\end{align*}
and hence $\|\tilde{x}_k^\prime\|^2 = \langle x^{\prime\prime}_1, z^k g_k\rangle =0$ for any $k\geq 1$, which ensures that $h^\prime\in E_2$. On the other hand, given $x_2^{\prime\prime}\in E_2$, there exists $h(z)=\sum_{n=0}^\infty z^n\tilde{x}_n\in H^2_{E_1}(\T)$ such that
$$
U\begin{bmatrix}
\Phi_1 h\\0
\end{bmatrix}=\begin{bmatrix}
    \Phi_2x_2^{\prime\prime} \\0
\end{bmatrix}.
$$
If $\tilde{x}_k\neq 0$ for some $k\geq 1$, and if $U\begin{bmatrix}
 \Phi_1 \tilde{x}_k \\0
\end{bmatrix}=\begin{bmatrix}
    \Phi_2 h^{\prime\prime}\\ 0
\end{bmatrix}$
for some $h^{\prime\prime}\in H^2_{E_2}(\T)$, then $\|\tilde{x}_k\|^2=\langle \Phi_1 h, \Phi_1z^k\tilde{x}_k\rangle$ implies
\begin{align*}
\|\tilde{x}_k\|^2 & =\left\langle U\begin{bmatrix}
    \Phi_1 h \\0
\end{bmatrix}, U\begin{bmatrix}
    z^k\Phi_1 \tilde{x}_k \\0
\end{bmatrix} \right\rangle = \left\langle \begin{bmatrix}
    \Phi_2 x_2^{\prime\prime} \\0
\end{bmatrix}, \begin{bmatrix}
    z^k\Phi_2 h^{\prime\prime} \\0
\end{bmatrix} \right\rangle = \langle x_2^{\prime\prime}, z^k h^{\prime\prime}\rangle = 0,
\end{align*}
thereby implying $h\in E_1$. The above information allows us to define $U_1: E_1\raro E_2$ by
\[
U_1x_1=x_2,
\]
whenever
\[
U\begin{bmatrix}
    \Phi_1 x_1 \\0
\end{bmatrix}=\begin{bmatrix}
    \Phi_2 x_2\\0
\end{bmatrix}.
\]
It is easy to check that $U_1$ is a well-defined, surjective as well as isometric linear map between $E_1$ and $E_2$, and therefore, is unitary. Thus, $\dim E_1 = \dim E_2$.

\textsf{Step 3:} Let us now assume that both $\clm_1$ and $\clm_2$ are Type II, with the canonical representations $\clm_j=\langle \clg_j \rangle\oplus(\Phi_j H_{E_j}^2(\T)\oplus\{0\})$, $j=1,2$, as described in the statement of this theorem, and, like in the previous part of the proof, $U:\clm_1\to \clm_2$ is the unitary operator satisfying the intertwining relation $U B^E_{\sigma_1, e^{i\theta_1}}{\big|_{\clm_1}} = B^E_{\sigma_2, e^{i\theta_2}}{\big|_{\clm_2}} U$. Suppose now for a given $h_1\in H^2_{E_1}(\T)$,
$$
U\begin{bmatrix}
    \Phi_1h_1\\ 0
\end{bmatrix}
=\begin{bmatrix}
    \Phi_2h_2\\0
\end{bmatrix}+\begin{bmatrix}
    g_2^\prime\\ y_2^\prime
\end{bmatrix}
$$
for some $h_2\in H^2_{E_2}(\T)$ and $\begin{bmatrix}
    g_2^\prime\\ y_2^\prime
\end{bmatrix}\in \langle \clg_2 \rangle$. In particular, we have
$$
\|h_1\|^2=\|h_2\|^2+\|g_2^\prime\|^2+\|y_2^\prime\|^2.
$$
At the same time, we have
$$
U\begin{bmatrix}
    z\Phi_1h_1\\0
\end{bmatrix}=UB^E_{\sigma_1, e^{i\theta_1}}\begin{bmatrix}\Phi_1h_1 \\0\end{bmatrix}=B^E_{\sigma_2, e^{i\theta_2}}U\begin{bmatrix}
 \Phi_1h_1\\0
\end{bmatrix}=\begin{bmatrix}
z\Phi_2h_2\\0
\end{bmatrix}+\begin{bmatrix}
zg_2^\prime+\sigma_2y_2^\prime\\ e^{i\theta_2}y_2^\prime
\end{bmatrix},
$$
and consequently,
$$
\|h_1\|^2=\|h_2\|^2+\|g_2^\prime\|^2+\|y_2^\prime\|^2+\sigma_2^2\|y_2^\prime\|^2.
$$
Therefore, $y_2^\prime=0$, and then
\[
\begin{bmatrix} g_2^\prime \\y_2^\prime
\end{bmatrix} \in \clm_2\cap(H^2_{E_2}(\T)\oplus\{0\})=\Phi_2H^2_{E_2}(\T)\oplus\{0\},
\]
forcing $g_2^\prime = 0$ as well. Therefore,
$$
U(\Phi_1H^2_{E_1}(\T)\oplus\{0\})\subseteq \Phi_2 H^2_{E_2}(\T)\oplus\{0\}.
$$
Using exactly similar lines of argument for $U^*$, we get
$$
U^*(\Phi_2H^2_{E_2}(\T)\oplus\{0\}) \subseteq \Phi_1 H^2_{E_1}(\T)\oplus\{0\},
$$
and hence $U$ is a unitary mapping between $\Phi_1H^2_{E_1}(\T)\oplus\{0\}$ and $\Phi_2H^2_{E_2}(\T)\oplus\{0\}$, such that
$$
UB^E_{\sigma_1, e^{i\theta_1}}\big|_{\Phi_1H^2_{E_1}(\T)\oplus\{0\}}=B^E_{\sigma_2, e^{i\theta_2}}\big|_{\Phi_2H^2_{E_2}(\T)\oplus\{0\}}U.
$$
Looking at the \textsf{Step 2}, where $\clm_1$ and $\clm_2$ both are taken to be Type I, it readily follows that for any $x_1\in E_1$, $U\begin{bmatrix}
    \Phi_1x_1 \\0
\end{bmatrix}=\begin{bmatrix}
    \Phi_2x_2 \\0
\end{bmatrix}$
for some $x_2\in E_2$, and conversely, for any $x_2\in E_2$ there exists $x_1\in E_1$ such that the above equality holds. This now guarantees the existence of a unitary map $U_E: E_1\raro E_2$, defined as follows:
$$
U_E(x_1)=x_2,
$$
whenever
\[
U\begin{bmatrix}\Phi_1x_1 \\0
\end{bmatrix}=\begin{bmatrix}\Phi_2x_2 \\0
\end{bmatrix},
\]
for $x_1\in E_1, x_2\in E_2$. For a given $\begin{bmatrix}
g \\ y
\end{bmatrix}\in \langle \clg_1 \rangle$, suppose now
\[
U \begin{bmatrix} g \\ y\end{bmatrix} = \begin{bmatrix} \Phi_2 h^\prime \\ 0\end{bmatrix}+\begin{bmatrix} g^\prime \\ y^\prime\end{bmatrix},
\]
for some $h^\prime\in H^2_{E_2}(\T)$ and $\begin{bmatrix}
    g^\prime\\ y^\prime
\end{bmatrix}\in \langle \clg_2 \rangle$. This implies
\begin{equation}\label{Brown_5}
\|g\|^2+\|y\|^2=\|h^\prime\|^2+\|g^\prime\|^2+\|y^\prime\|^2.
\end{equation}
Now there exists $h\in H^2_{E_1}(\T)$ such that
$U\begin{bmatrix}
    \Phi_1h\\0
\end{bmatrix}=\begin{bmatrix}
    \Phi_2 h^\prime\\0
\end{bmatrix}.$
As a result,
\[
U\begin{bmatrix} g \\ y\end{bmatrix}=U\begin{bmatrix} \Phi_1 h \\ 0\end{bmatrix}+\begin{bmatrix} g^\prime \\ y^\prime\end{bmatrix},
\]
that is,
\[
U\begin{bmatrix} g - \Phi_1 h \\ y\end{bmatrix} = \begin{bmatrix} g^\prime \\ y^\prime\end{bmatrix}.
\]
Consequently,
\[
\|g\|^2+\|h\|^2+\|y\|^2 =\left\|U\begin{bmatrix} g-\Phi_1h\\ y\end{bmatrix}\right\|^2 = \left\|\begin{bmatrix} g^\prime \\ y^\prime\end{bmatrix}\right\|^2=\|g^\prime\|^2+\|y^\prime\|^2,
\]
which, combined with (\ref{Brown_5}) yields $h=h^\prime=0$. In other words, $U(\langle \clg_1 \rangle)\subseteq \langle \clg_2 \rangle$. Imitating this argument line by line with $U$ replaced by $U^*$, it immediately follows that $U^*(\langle \clg_2 \rangle)\subseteq \langle \clg_1 \rangle$. We therefore conclude that there exists a unitary map $U_\clg=U\big|_{\langle \clg_1 \rangle}$ between $\langle \clg_1 \rangle$ and $\langle \clg_2 \rangle$. Now consider any $\begin{bmatrix}
    g_j \\ y_j
\end{bmatrix}\in \langle \clg_j \rangle$, $j=1, 2$, such that $U\begin{bmatrix}
    g_1\\y_1
\end{bmatrix}=\begin{bmatrix}
    g_2\\ y_2
\end{bmatrix}$, and recall that
\begin{equation}\label{Brown_eq}
z g_j + \sigma_jy_j = e^{i\theta_j} g_j + \Phi_jx^\prime_j,
\end{equation}
for a unique $x_j^\prime\in E_j$. Then
\[
\left\langle B^E_{\sigma_1, e^{i\theta_1}}\begin{bmatrix} g_1 \\ y_1\end{bmatrix}, U^*\begin{bmatrix} g_2 \\ y_2\end{bmatrix}\right\rangle= \left\langle B^E_{\sigma_2, e^{i\theta_2}}U\begin{bmatrix} g_1 \\ y_1\end{bmatrix}, \begin{bmatrix} g_2 \\ y_2\end{bmatrix}\right\rangle.
\]
The above identity implies
\[
\left\langle \begin{bmatrix} zg_1+\sigma_1y_1\\ e^{i\theta_1}y_1\end{bmatrix}, \begin{bmatrix} g_1 \\ y_1\end{bmatrix}\right\rangle=\left\langle \begin{bmatrix} zg_2+\sigma_2y_2\\ e^{i\theta_2}y_2\end{bmatrix}, \begin{bmatrix} g_2 \\ y_2\end{bmatrix}\right\rangle,
\]
or, equivalently,
\[
\left\langle e^{i\theta_1}\begin{bmatrix} g_1\\ y_1\end{bmatrix}+\begin{bmatrix}\Phi_1x_1^\prime \\ 0\end{bmatrix}, \begin{bmatrix} g_1\\ y_1\end{bmatrix}\right\rangle = \left\langle  e^{i\theta_2}\begin{bmatrix} g_2\\ y_2\end{bmatrix}+\begin{bmatrix}\Phi_2x_2^\prime \\ 0\end{bmatrix}, \begin{bmatrix} g_2\\ y_2\end{bmatrix} \right\rangle,
\]
and so
\[
e^{i\theta_1} \left\|\begin{bmatrix} g_1\\ y_1\end{bmatrix}\right\|^2 = e^{i\theta_1} \left\|U\begin{bmatrix} g_1\\ y_1\end{bmatrix}\right\|^2=e^{i\theta_2}\left\|\begin{bmatrix} g_2\\ y_2\end{bmatrix}\right\|^2.
\]
It follows that $e^{i\theta_1}=e^{i\theta_2}$. Since $\theta_1, \theta_2\in[0, 2\pi)$, we conclude that $\theta_1=\theta_2$. Let us set again $\theta=\theta_1=\theta_2$. Using this information and $(\ref{Brown_eq})$, we finally observe that
\begin{align*}
U\begin{bmatrix}
    \Phi_1x_1^\prime\\0
\end{bmatrix}+e^{i\theta}U\begin{bmatrix}
    g_1 \\ y_1
\end{bmatrix} = U\begin{bmatrix}
    zg_1+\sigma_1 y_1\\e^{i\theta}y_1
\end{bmatrix} = UB^E_{\sigma_1, e^{i\theta}}\begin{bmatrix}g_1\\ y_1\end{bmatrix} = B^E_{\sigma_2, e^{i\theta}}\begin{bmatrix}
    g_2 \\ y_2
\end{bmatrix}.
\end{align*}
As
\begin{align*}
B^E_{\sigma_2, e^{i\theta}}\begin{bmatrix}
    g_2 \\ y_2
\end{bmatrix} = \begin{bmatrix}
    zg_2+\sigma_2 y_2\\e^{i\theta}y_2
\end{bmatrix} = \begin{bmatrix}
    \Phi_2x_2^\prime\\0
\end{bmatrix}+e^{i\theta}\begin{bmatrix}
    g_2 \\ y_2
\end{bmatrix},
\end{align*}
it follows that $U\begin{bmatrix}
    \Phi_1x_1^\prime\\0
\end{bmatrix}=\begin{bmatrix}\Phi_2x_2^\prime\\0
\end{bmatrix}$. As a consequence, $U_E x_1^\prime=x_2^\prime$. This completes the proof.
\end{proof}

In the following section, we illustrate this theorem in the case $E = \mathbb{C}$.

\section{The scalar case}\label{sec: scalar}

The purpose of this section is to recover the representations of invariant subspaces of Brownian shifts on $H^2(\T)$, as obtained by Agler and Stankus (see \cite[pp. 21-24]{Agler-Stankus}). More specifically, in the particular case of $E = \mathbb{C}$, by using Theorem \ref{th1-br} and combining Theorems \ref{th2-br} and \ref{thm: type II conv} with Remark \ref{rem: boundary}, we retrieve the representations of invariant subspaces of Agler and Stankus:

\begin{theorem}\label{th-scalar Agler}
Let $\clm$ be a nonzero closed subspace of $H^2(\T) \oplus \mathbb{C}$. Then $\clm$ is invariant under $\br$ if and only if it admits one of the following representations:
\[
\clm = \vp H^2(\T) \oplus \{0\},
\]
for some inner function $\vp \in H^\infty(\T)$, or
\[
\clm = \mathbb{C}\begin{bmatrix} g\\ 1\end{bmatrix} \oplus \left(\psi H^2(\T)\oplus\{0\}\right),
\]
where $\psi \in H^\infty(\T)$ is an inner function with the condition that $\psi(e^{i\theta})$ exists, and
\[
g = \sigma \left(\frac{\overline{\psi(e^{i\theta})} \psi - 1}{z-e^{i\theta}}\right)\in H^2(\T).
\]	
\end{theorem}
\begin{proof}
We start with the proof of ``only if" part. Suppose $\clm\subseteq H^2(\T)\oplus\{0\}$. If $B_{\sigma, e^{i\theta}}(\clm)\subseteq \clm$, then setting $E=\C$ in Theorem \ref{th1-br}, we find that $E_1 = \C$ (note that $\clm \neq \{0\}$), and
\[
\clm=\vp H^2(\T)\oplus\{0\},
\]
for some inner function $\vp \in H^\infty(\T)$. Next, assume that $\clm\nsubseteq H^2(\T)\oplus\{0\}$. Under the assumption $E=\C$, Theorem \ref{th2-br} asserts that
\[
\clm = \cld_\clm \oplus (\psi H^2(\T)\oplus\{0\}),
\]
for some inner  function $\psi \in H^\infty(\T)$, where
\[
\cld_\clm= \clm \ominus (\clm\cap (H^2(\T)\oplus \{0\})) \neq \{0\}.
\]
It now remains to show that any element of $\cld_\clm$ is a scalar multiple of $\begin{bmatrix}
    g \\ 1
\end{bmatrix}$, $g$ as given in the statement of the theorem.
Making use of Theorem \ref{th2-br} and Remark \ref{rem: boundary}, we see that for any nonzero $\begin{bmatrix}
    g_1 \\ \beta
\end{bmatrix}\in \cld_\clm$, there exists unique $\alpha\in \C$ with $|\alpha|=\sigma|\beta|>0$ such that
\begin{equation}\label{rep_g}
g_1=\frac{\psi \alpha-\sigma\beta }{z-e^{i\theta}},
\end{equation}
and
\[
\lim_{r\to 1-}\psi(re^{i\theta})\alpha = \sigma\beta.
\]
Since $|\alpha|=\sigma|\beta|> 0$, $|\psi(e^{i\theta})|=1$. As a result, $$\alpha = \sigma \overline{\psi(e^{i\theta})}\beta.$$ Using this value of $\alpha$ in \eqref{rep_g}, we deduce that
\begin{equation*}
g_1=\beta\sigma\left( \frac{\overline{\psi(e^{i\theta})}\psi -1 }{z-e^{i\theta}}\right),
\end{equation*}
which means $\begin{bmatrix}
    g_1\\ \beta
\end{bmatrix}=\beta\begin{bmatrix}
    g\\ 1
\end{bmatrix}$.
The ``if'' part follows trivially from Theorems \ref{th1-br} and \ref{thm: type II conv}.
\end{proof}

In the scalar case, the unitary equivalence of invariant subspaces established in Theorem \ref{thm: unit equiv} takes the following form:

\begin{theorem}\label{complex-br}
Fix angles $\theta_1, \theta_2 \in [0, 2 \pi)$ and covariances $\sigma_1, \sigma_2 > 0$. Let $\clm_1$ and $\clm_2$ be nonzero closed invariant subspaces of the Brownian shifts $B_{\sigma_1, e^{i\theta_1}}$ and $B_{\sigma_2, e^{i\theta_2}}$ in $H^2(\T)\oplus\C$, respectively. Then
\[
B_{\sigma_1, e^{i\theta_1}}{\big|_{\clm_1}} \cong B_{\sigma_2, e^{i\theta_2}}{\big|_{\clm_2}},
\]
if and only if any one of the following conditions is true:
\begin{itemize}
\item[(1)]
Both $\clm_1$ and $\clm_2$ are Type I.
\item[(2)]
Both $\clm_1$ and $\clm_2$ are Type II, along with the facts that
\[
\theta_1 = \theta_2,
\]
and
\[
\sigma_2^2(1+\|g_1\|^2)=\sigma_1^2(1+\|g_2\|^2),
\]
where $\clm_j = \mathbb{C}\begin{bmatrix}g_j\\ 1\end{bmatrix} \oplus \left(\varphi_j H^2(\T)\oplus\{0\}\right)$ is the canonical representation of $\clm_j$, and $g_j = \sigma_j\left(\frac{\overline{\vp_j(e^{i\theta_j})} \vp_j - 1}{z - e^{i\theta_j}}\right)$ for $j=1,2$.
\end{itemize}
\end{theorem}
\begin{proof}
It is clear that we can now assume $E_1=E_2=\C$ and $\clg_1=\C\begin{bmatrix}
    g_1 \\1
\end{bmatrix}$, $\clg_2=\C\begin{bmatrix}
    g_2\\ 1
\end{bmatrix}$ in the statement of Theorem \ref{thm: unit equiv}. All we have to verify is that condition (2) of Theorem \ref{thm: unit equiv} becomes equivalent to condition (2) of this theorem for $E=\C$. In both cases $\theta_1=\theta_2$ is common in condition (2), so we only need to concentrate on the rest. Suppose
$$U_\clg\begin{bmatrix}
    g_1\\1
\end{bmatrix}=\alpha\begin{bmatrix}
    g_2\\1
\end{bmatrix}$$ for some $\alpha\in\C$. It is evident that
$1+\|g_1\|^2=|\alpha|^2(1+\|g_2\|^2)$, and at the same time, according to condition (2) of Theorem \ref{thm: unit equiv}
$$U_E\left(\sigma_1\overline{\vp_1(e^{i\theta})}\right)=\alpha\sigma_2\overline{\vp_2(e^{i\theta})},$$ which gives $|\alpha|=\sigma_1/\sigma_2$. As a result, \begin{equation}\label{Newcor}\sigma_2^2(1+\|g_1\|^2)=\sigma_1^2(1+\|g_2\|^2).
\end{equation}
Conversely, if we start by assuming (\ref{Newcor}), then it is immediately seen that $$U_\clg\begin{bmatrix}
    g_1\\1
\end{bmatrix}=\frac{\sigma_1}{\sigma_2}\begin{bmatrix}
    g_2\\1
\end{bmatrix},$$
and $$
U_E\left(\overline{\vp_1(e^{i\theta})}\right)=\overline{\vp_2(e^{i\theta})}
$$
define surjective isometries $\langle\clg_1\rangle\raro \langle\clg_2\rangle$ and $E_1\raro E_2$, respectively, and therefore, they are unitaries. Our proof is therefore done.
\end{proof}

This result was previously obtained in \cite[Theorem 2.1]{DDS}.

\section{Reducing subspaces}\label{sec: reducing}

Given a bounded linear operator $A$ on a Hilbert space $\clh$, a closed subspace $\cls \subseteq \clh$ is said to be \textit{reducing} for $A$ (or $A$-reducing) if $\mathcal{S}$ is invariant under both $A$ and $A^*$. Our goal here is to classify all reducing subspaces of Brownian shifts.

Recall that a closed subspace $\clm$ of $H^2_E(\T)$ is reducing for $S_E$ if and only if there exists a closed subspace $F \subseteq E$ such that $\clm = H^2_F(\T)$ (cf. \cite[Theorem 3.22]{Radj-Rosen}). For the question of reducing subspaces of a Brownian shift $\br^E$ on $H_E^2(\T)\oplus E$, the answer is as follows:

\begin{theorem}
Let $\clm$ be a closed subspace of $H_E^2(\T)\oplus E$. Then $\clm$ is a reducing subspace for $\br^E$ if and only if there exists a closed subspace $G$ of $E$ such that
\[
\clm = H_{G}^2(\T)\oplus G.
\]
\end{theorem}
\begin{proof}
Given the adjoint operator $\left(B^E_{\sigma, e^{i\theta}}\right)^* =\begin{bmatrix}S_E^* & 0\\ \sigma i^*_E & e^{-i\theta}I_E \end{bmatrix}$, the sufficiency is straightforward. Suppose $\clm$ is a reducing subspace for $\br^E$. Let $\begin{bmatrix}f \\ x \end{bmatrix} \in \clm$. Then
\[
\begin{bmatrix}
f\\\sigma^2 x + x
\end{bmatrix} = \left(\br^E\right)^*\br^E \begin{bmatrix}
f\\ x \end{bmatrix} \in \clm,
\]
and hence
\[
\begin{bmatrix}
0\\ \sigma^2 x
\end{bmatrix} = \begin{bmatrix} f\\ \sigma^2 x + x
\end{bmatrix} -  \begin{bmatrix}
f\\
x
\end{bmatrix} \in \clm.
\]
As $\sigma>0$, this implies $\begin{bmatrix} 0\\ x\end{bmatrix}\in \clm$. On the other hand, since $\br^E \begin{bmatrix} 0\\ x\end{bmatrix} =  \begin{bmatrix}\sigma x \\e^{i\theta}x\end{bmatrix} \in \clm$, it follows that
\[
\begin{bmatrix}\sigma x\\ e^{i\theta}x\end{bmatrix} - e^{i\theta}\begin{bmatrix} 0\\ x\end{bmatrix} = \sigma  \begin{bmatrix} x \\ 0\end{bmatrix} \in \clm,
\]
and hence, $\begin{bmatrix} x\\ 0\end{bmatrix} \in \clm$, and also, $\begin{bmatrix}f \\ 0 \end{bmatrix} = \begin{bmatrix}f \\ x \end{bmatrix} -\begin{bmatrix} 0\\ x\end{bmatrix} \in \clm$. Therefore, for each $\begin{bmatrix}f \\ x \end{bmatrix} \in \clm$, we have
\begin{equation}\label{reducing_vBS}
\Big\{\begin{bmatrix} x\\ 0\end{bmatrix}, \begin{bmatrix}  0\\ x\end{bmatrix}, \begin{bmatrix}f \\ 0 \end{bmatrix} \Big\} \subseteq \clm.
\end{equation}
Write $f = \sum_{n=0}^\infty \hat{f}(n) z^n$, where $\hat{f}(n) \in E$ for all $n \geq 0$. Then
\[
\left(\br^E\right)^*\begin{bmatrix}f \\ 0 \end{bmatrix} = \begin{bmatrix}S_E^*f \\ \sigma \hat{f}(0) \end{bmatrix} \in \clm,
\]
and hence (\ref{reducing_vBS}) implies $\Big\{\begin{bmatrix}\hat{f}(0) \\ 0 \end{bmatrix}, \begin{bmatrix}0 \\\hat{f}(0)\end{bmatrix}, \begin{bmatrix}S_E^*f \\ 0 \end{bmatrix}\Big\} \subseteq \clm$. Assuming that $$\Big\{\begin{bmatrix}\hat{f}(m-1) \\ 0 \end{bmatrix}, \begin{bmatrix}0 \\\hat{f}(m-1)\end{bmatrix}, \begin{bmatrix}{S_E^*}^mf \\ 0 \end{bmatrix}\Big\} \subseteq \clm$$
for any $m\geq 1$, we observe that $\left(\br^E\right)^*\begin{bmatrix}{S_E^*}^mf \\ 0 \end{bmatrix}=\begin{bmatrix}{S_E^*}^{m+1}f \\ \sigma \hat{f}(m)\end{bmatrix}\in\clm$, and again from (\ref{reducing_vBS}), it follows that
$$\Big\{\begin{bmatrix}\hat{f}(m) \\ 0 \end{bmatrix}, \begin{bmatrix}0 \\\hat{f}(m)\end{bmatrix}, \begin{bmatrix}{S_E^*}^{m+1}f \\ 0 \end{bmatrix}\Big\} \subseteq \clm.$$
The principle of mathematical induction now yields
\[
\Big\{\begin{bmatrix}\hat{f}(n) \\ 0 \end{bmatrix}, \begin{bmatrix}0 \\\hat{f}(n)\end{bmatrix}: n \geq 0 \Big\} \subseteq \clm.
\]
Set $G=\langle G_0\rangle$, where
\[
G_0 = \left\{x, \hat{f}(n):  \begin{bmatrix}f \\ x \end{bmatrix} \in \clm, n\geq 0\right\}.
\]
Therefore, $G$ is a closed subspace of $E$. Moreover, we let
\[
G_1 = \left\{\begin{bmatrix} 0\\ x\end{bmatrix}, \begin{bmatrix} 0\\ \hat{f}(n) \end{bmatrix}:  \begin{bmatrix}f \\ x \end{bmatrix} \in \clm, n\geq 0\right\},
\]
and
\[
G_2 = \left\{\begin{bmatrix} x\\ 0 \end{bmatrix}, \begin{bmatrix} \hat{f}(n)\\ 0\end{bmatrix}:  \begin{bmatrix}f \\ x \end{bmatrix} \in \clm, n\geq 0\right\}.
\]
Then $\langle G_1\rangle = \{0\}\oplus G$ and $\langle G_2\rangle = G \oplus \{0\}$. Now $\langle G_2\rangle \subseteq \clm$ implies $H^2_G(\T)\oplus \{0\}\subseteq \clm$. This and $\{0\}\oplus G = \langle G_1\rangle \subseteq \clm$ yields
$$
H^2_G(\T)\oplus G \subseteq \clm.
$$
For the reverse inclusion, we pick $\begin{bmatrix}f \\ x \end{bmatrix} \in \clm$, and assume $\begin{bmatrix}f \\ x \end{bmatrix} \perp (H^2_G(\T)\oplus G)$. In particular, $\begin{bmatrix}f \\ x \end{bmatrix} \perp \{0\} \oplus G$ implies $x = 0$. Similarly, $\begin{bmatrix} \hat{f}(n)\\ 0\end{bmatrix} \perp G \oplus \{0\}$ for all $n \geq 0$ implies that $f = 0$. This proves $H^2_G(\T)\oplus G = \clm$.
\end{proof}

Recall that a bounded linear operator $A$ on $\mathcal{H}$ is irreducible if there is no nontrivial closed subspace of $\clh$ that reduces $A$. The following is now straightforward: $\br$ on $H^2(\T) \oplus \C$ is irreducible for all angle $\theta \in [0, 2\pi)$ and covariance $\sigma > 0$. This was previously observed in \cite[Proposition 4.3]{DDS}.

\section{Examples}\label{sec: examples}

Let $\clm \subseteq H^2_E(\T)$ be a nonzero closed invariant subspace of the unilateral shift operator $S_E$. Then the restriction ${S_E}|_{\clm}$ is itself a shift operator. In other words, up to unitary equivalence, the restriction of a shift operator to an invariant subspace is again a shift operator. Thus, from the perspective of operator theory, invariant subspaces of shift operators do not yield fundamentally new operators. The Brownian shift on $H^2(\T)$ exhibits a similar property: up to unitary equivalence, the restriction of a Brownian shift on $H^2(\T)$ to a Type II invariant subspace yields another Brownian shift on $H^2(\T)$ \cite[Proposition 5.75]{Agler-Stankus} (see also \cite{DDS}).

At the other extreme, the situation is different for the Bergman and Dirichlet shifts on the unit disc. A rigidity result of Richter shows that the restrictions of the Bergman or Dirichlet shifts to their invariant subspaces always produce unitarily non-equivalent operators, that is, each invariant subspace gives rise to a distinct operator \cite{Richter} (also see \cite{Ron, DS}).

In this section, we present two concrete examples that illustrate the nuanced behavior of Brownian shifts. In the first example, we show that the restriction of a Brownian shift to an invariant subspace does not necessarily yield another Brownian shift:

\begin{example}\label{BSonVHS-ex1}
Consider a Brownian shift $\br^E$ acting on $H^2_E(\T) \oplus E$ for some Hilbert space $E$ with $\dim(E)\geq 2$. Pick a nonzero closed subspace $E_1$ of $E$ such that
\[
\dim E_1 \neq \dim E.
\]
By the upper triangular representation of Brownian shifts, it follows that $\he\oplus E_1$ is a Type II invariant subspace for $\br^E$. If possible, suppose there exist $\sigma^\prime>0$, $\theta^\prime\in[0, 2\pi)$ and a Hilbert space $G$ such that
\begin{equation}\label{Ex-eq-2}
\br^E\big|_{\he\oplus E_1}\cong B_{\sigma^\prime, e^{i\theta^\prime}}^G.
\end{equation}
Consider the Brownian shift $B_{\sigma^\prime, e^{i\theta^\prime}}^K$ on $H^2_K(\T) \oplus K$, where
\[
K = E \oplus G.
\]
Clearly, there exists a unitary $U_1: \he\oplus E_1\raro H^2_{E\oplus \{0\}}(\T)\oplus (E_1\oplus \{0\})$ such that
\[
U_1\br^E\big|_{\he\oplus E_1}=\br^K\big|_{H^2_{E\oplus \{0\}}(\T)\oplus (E_1\oplus \{0\})}U_1.
\]
Similarly, there is a unitary $U_2:H^2_G(\T)\oplus G\raro H^2_{ \{0\} \oplus G}(\T)\oplus ( \{0\}\oplus G)$ such that
\[
U_2 B_{\sigma^\prime, e^{i\theta^\prime}}^G= B_{\sigma^\prime, e^{i\theta^\prime}}^K\big|_{H^2_{ \{0\} \oplus G}(\T)\oplus ( \{0\}\oplus G)}U_2.
\]
As a result, \eqref{Ex-eq-2} implies
\begin{equation}\label{Ex-eq-3}
\br^K\big|_{H^2_{E\oplus \{0\}}(\T)\oplus (E_1\oplus \{0\})} \cong B_{\sigma^\prime, e^{i\theta^\prime}}^K\big|_{H^2_{ \{0\} \oplus G}(\T)\oplus ( \{0\}\oplus G)}.
\end{equation}
On the other hand, observe that
\[
H^2_{E\oplus \{0\}}(\T)\oplus (E_1\oplus \{0\})=\left(H^2_{E\oplus \{0\}}(\T)\oplus \{0\}\right)\oplus\left(\{0\}\oplus(E_1\oplus \{0\})\right),
\]
and given any element $\begin{bmatrix} 0\\e_1 \end{bmatrix}$, with $e_1\in E_1\oplus\{0\}$, we have
\[
0=\frac{\sigma e_1-\sigma e_1}{z-e^{i\theta}}.
\]
Therefore, it becomes apparent that the subspace $H^2_{E\oplus \{0\}}(\T)\oplus (E_1\oplus \{0\})$ of $H^2_K(\T)\oplus K$ is a  Type II invariant subspace for $\br^K$, and so will be
$$
H^2_{ \{0\} \oplus G}(\T)\oplus ( \{0\}\oplus G)
=\left(H^2_{ \{0\} \oplus G}(\T)\oplus\{0\}\right)\oplus\left(\{0\}\oplus ( \{0\}\oplus G)\right)
$$
for $B_{\sigma^\prime, e^{i\theta^\prime}}^K$. However, in view of \eqref{Ex-eq-3}, condition (2) of Theorem \ref{thm: unit equiv} guarantees the existence of unitary operators $U^\prime_E$ and $U^\prime_{E_1}$ such that
\[
U^\prime_E(E\oplus \{0\})=\{0\}\oplus G
\]
and
\[
U^\prime_{E_1}(\{0\}\oplus (E_1\oplus \{0\}))=\{0\}\oplus (\{0\}\oplus G).
\]
Clearly, these two unitary operators $U^\prime_E$ and $U^\prime_{E_1}$ induce two other unitary operators $U_E$ and $U_{E_1}$ such that
\[
U_E(E)= G \text{ and } U_{E_1}(E_1)=G.
\]
Hence, $U_{E_1}^*U_E$ defines a unitary operator between $E$ and $E_1$, leading to the conclusion:
\[
\dim E = \dim E_1,
\]
which is contrary to our assumption. Therefore, the operator $\br^E\big|_{\he\oplus E_1}$ is never unitarily equivalent to any other Brownian shift.
\end{example}

In the following example, we show that the restriction of a Brownian shift to an invariant subspace may result in operator that is unitarily equivalent to another Brownian shift.

\begin{example}\label{BSonVHS-ex2}
Consider the Brownian shift $B_{\sigma, 1}^{\C^2}$ on $H^2_{\C^2}(\T)\oplus\C^2$. We use the row-vector notation $\begin{bmatrix}f_1 & f_2\end{bmatrix}$, where $f_1, f_2\in H^2(\T)$, to denote an element of $H^2_{\C^2}(\T)$. Note that, in particular, either $f_1$ or $f_2$ may be a constant in $\mathbb{C}$. However, for $\begin{bmatrix}
  0 &0
\end{bmatrix}$, we will simply write $0$ when there is no chance of confusion. Now, we begin by considering the one-dimensional subspace of $\C^2$ as
\[
E_1=\langle[1, 1]\rangle.
\]
Also, consider the inner function $\Phi\in H^\infty_{\clb(E_1, \C^2)}(\T)$, defined by $\Phi(z)= \Phi_1(z)\big|_{E_1}$, where 
\[
\Phi_1(z)=\begin{bmatrix}
z &0\\0& 1
\end{bmatrix}.
\]
Note that $\Phi_1 \in H^\infty_{\clb(\C^2)}(\T)$ is an inner function. Since we can write $\begin{bmatrix}\sigma & 0\end{bmatrix}$ as
$$
\begin{bmatrix}\sigma & 0\end{bmatrix} = \frac{\Phi\begin{bmatrix}\sigma & \sigma\end{bmatrix}-\sigma\begin{bmatrix}1 & 1\end{bmatrix}}{z-1}\in H^2_{\C^2}(\T),
$$
by recalling the definition of $\mathcal{G}_\Phi$ from \eqref{eqn: G Phi}, it follows that
$$
\mathbb{g}=
\begin{bmatrix}
  \begin{bmatrix}\sigma & 0\end{bmatrix}\\\\
  \begin{bmatrix}1 & 1\end{bmatrix}
\end{bmatrix}
$$
is a member of $\mathcal{G}_\Phi$. Therefore, Theorem \ref{thm: type II conv} ensures that
$$
\clm=(\Phi H^2_{E_1}(\T)\oplus\{0\})\oplus\langle \mathbb{g}\rangle,
$$
is a Type II invariant subspace of $B_{\sigma, 1}^{\C^2}$. At this point, we turn our attention to the Brownian shift
\[
B_{\sigma^\prime, 1}^{\C\oplus\{0\}}\in\clb\left(H^2_{\C\oplus \{0\}}(\T)\oplus\left(\C\oplus\{0\}\right)\right),
\]
for some $\sigma^\prime > 0$. We claim that
\begin{equation}\label{Ex-eq-1}
B_{\sigma, 1}^{\C^2}\big|_\clm\cong B_{\sigma^\prime, 1}^{\C\oplus\{0\}}=B_{\sigma^\prime, 1}^{\C^2}\big|_{H^2_{\C\oplus \{0\}}(\T)\oplus\left(\C\oplus\{0\}\right)},
\end{equation}
where
\[
\sigma^\prime=\frac{\sigma\sqrt{2}}{\sqrt{2+\sigma^2}}>0.
\]
It is evident that $H^2_{\C\oplus \{0\}}(\T)\oplus\left(\C\oplus\{0\}\right)$ is a Type II invariant subspace of $B_{\sigma^\prime, 1}^{\C^2}\in\clb(H^2_{\C^2}(\T)\oplus\C^2)$, if we recognize that
$$
H^2_{\C\oplus \{0\}}(\T)\oplus\left(\C\oplus\{0\}\right)
=\left(H^2_{\left\langle\begin{bmatrix}
    1& 0
\end{bmatrix}\right\rangle}(\T)\oplus\{0\}\right)\oplus\C\begin{bmatrix}
    \begin{bmatrix}
        0 & 0
    \end{bmatrix}\\\\
    \begin{bmatrix}
        1 & 0
    \end{bmatrix}
\end{bmatrix},
$$
and that
$$
\begin{bmatrix}0 & 0\end{bmatrix}=\frac{\begin{bmatrix}\sigma^\prime & 0\end{bmatrix}-\sigma^\prime\begin{bmatrix}1 & 0\end{bmatrix}}{z-1}.
$$
With all the above information, we are now able to define a unitary $U_\mathbb{g}$ by
$$
U_\mathbb{g}(\mathbb{g})=\sqrt{2+\sigma^2}\begin{bmatrix}
    \begin{bmatrix}
        0 & 0
    \end{bmatrix}\\\\
    \begin{bmatrix}
        1 & 0
    \end{bmatrix}
    \end{bmatrix},
$$
and another unitary $U_{E_1}$ by
$$
U_{E_1}\left(\begin{bmatrix}
    \sigma& \sigma
\end{bmatrix}\right)
=\sqrt{2+\sigma^2}\begin{bmatrix}
    \sigma^\prime & 0
\end{bmatrix}.
$$
The condition (2) of Theorem \ref{thm: unit equiv} now guarantees that (\ref{Ex-eq-1}) is true.
\end{example}

In light of the preceding two examples, it is now clear that Brownian shifts have distinct characteristics that are not commonly observed in shifts on Hilbert function spaces.

We conclude this paper with a remark: Let $T$ be a contraction on a Hilbert space $\clh$. We say that $T \in C_{\cdot 0}$ if
\[
SOT-\lim_{m \raro \infty} T^{*m}= 0.
\]
The contraction $T$ is said to satisfy the $C_{00}$-property if both $T$ and $T^*$ are in $C_{\cdot 0}$. In this case, we simply write $T \in C_{00}$. In \cite[Theorem 3.2]{DDS} we proved that $\frac{1}{\|\br\|}\br \in C_{00}$ for all covariance $\sigma > 0$ and angle $\theta \in [0, 2\pi)$. A similar technique establishes in the vector-valued case:
\[
\frac{1}{\|\br^E\|} \br^E \in C_{00}.
\]
This result is also comparable to the asymptotic properties observed in \cite{Chavan et al}. Finally, we remark that, following the computation of \cite[Section 3]{DDS}, one concludes that $\br^E$ is not similar to a contraction.

\vspace{0.1in}

\noindent\textbf{Acknowledgement:}
We are deeply grateful to Professor Jan Stochel for responding to our questions regarding the boundary behavior of inner functions in the present context. The first named author is supported by the National Board for Higher Mathematics (NBHM), India (Order No.: 0204/10/(18)/2023/R\&D-II/2791 dated 28 February, 2023). The research of the second named author is supported by the National Board for Higher Mathematics (NBHM), India (Order No: 0204/16(8)/2024/R\&D- II/6760, dated May 09, 2024). The research of the third named author is supported in part by TARE (TAR/2022/000063) by SERB, Department of Science \& Technology (DST), Government of India.

\end{document}